\documentclass{amsart}
    \usepackage{amsmath, amssymb, amsfonts, hyperref}
    \newtheorem{theorem}{Theorem}[section]
    
    \newtheorem{proposition}[theorem]{Proposition}
    \newtheorem{corollary}[theorem]{Corollary}
    \newtheorem{lemma}[theorem]{Lemma}

    \newtheorem{remark}[theorem]{Remark}

    \def\qed{\hfill $\Box$\medskip}
    
    \def\bM{{\mathbb M}}

    \def\RR{\mathbb{R}}

    \def\KK{{\mathbb{K}}}

    \def\HH{{\mathbb{H}}}

    \def\FF{{\mathbb F}}
    \def\CC{{\mathbb C}}

    \def\A{{\mathcal A}}

    \def\RR{{\mathbb R}}
    
\begin{document}
    \openup 1\jot

    \title{Birkhoff-James classification of finite-dimensional $C^*$-algebras}

 \author[B. Kuzma]{Bojan Kuzma}
    \address{University of Primorska, Glagolja\v{s}ka 8, SI-6000 Koper, Slovenia, and
    Institute of Mathematics, Physics, and Mechanics, Jadranska 19, SI-1000 Ljubljana, Slovenia.}
     \email{bojan.kuzma@upr.si}

    \author[S. Singla]{Sushil Singla}
    \address{University of Primorska, Glagolja\v{s}ka 8, SI-6000 Koper, Slovenia,}
     \email{ss774@snu.edu.in}


 \thanks{This work is
    supported in part by the Slovenian Research Agency (research program P1-0285 and research projects N1-0210, N1-0296 and J1-50000). }

    \begin{abstract}
        We classify real or complex finite-dimensional $C^*$-algebras and their underlying fields 
        from the 
        properties of 
        Birkhoff-James orthogonality. 
        Application to strong Birkhoff-James orthogonality preservers 
        is also given.
    \end{abstract}

    \maketitle

    \section{Introduction}
    Two vectors $u,v$ in a Hilbert space are orthogonal if their inner product is zero. This is equivalent to the fact that the distance from  $u$ to the one-dimensional subspace spanned by $v$ is attained at $0$. We can then
         define orthogonality in general normed spaces over the field $\FF\in\{\RR,\CC\}$  by
    $$u\perp v\quad\hbox{if}\quad \|u+\lambda v\|\ge \|u\|\text{ for all }\lambda \in \FF.$$ This relation, which was proposed by Birkhoff \cite{Birkhoff} and further developed by James \cite{james}, is known as \textit{Birkhoff-James (BJ for short) orthogonality}. It turns out that BJ orthogonality knows a lot about the norm. For example, James \cite[Theorem 1]{james}, by using the Kakutani-Bohnenblust characterization of inner-product spaces  \cite[Theorem 3]{Kakutani} and \cite{Bohneblust}   (see also a book by Amir~\cite{amir}), states that a normed space $V$ of dimension greater or equal to three is an inner product space if and only if BJ orthogonality is symmetric. See also the recent results \cite{Aram, guterman, simple, abelian} for more.

    Our main aim  is to
         show that BJ orthogonality knows everything about (real or complex) finite-dimensional $C^\ast$-algebras.      Namely, as we show in our main result below, two finite-dimensional $C^\ast$-algebras $\A_1$ and $\A_2$  are \emph{BJ-isomorphic} (i.e., there exists a bijection $\phi\colon\A_1\to\mathcal A_2$ with      $u\perp v\Longleftrightarrow\phi(u
    )\perp\phi(v)$) if and only if their underlying fields are the same and
         they are $\ast$-isomorphic.
         This is in spirit of Tanaka~\cite{tanaka3} who      characterized complex abelian  $C^*$-algebras among all (finite or infinite-dimensional) complex $C^*$-algebras in terms of geometric structure (which was first defined from BJ orthogonality \cite{tanaka1}). We have been informed by Tanaka \cite{private} that he has recently also found a similar classification within  finite-dimensional complex $C^*$-algebras.            We refer to a book by Goodearl~\cite{goodearl} for more on the theory of  real $C^*$-algebras, which  are similar to the complex ones.

    \begin{theorem}\label{maintheorem1}Suppose  $C^*$-algebras $\A_1$ and $\A_2$,  over the fields $\FF_1$ and $\FF_2$,
         are BJ-isomorphic. If $\A_1$ is finite-dimensional with $\dim\A_1\geq 2$, then:      \begin{enumerate}
        \item  $\FF_1=\FF_2$,
        \item $\A_1$ and $\A_2$ are isomorphic as $C^\ast$-algebras.
    \end{enumerate}
    \end{theorem}
    Note that (1) may fail if $\dim\A_1=1$, see~\cite[Example 2.2]{simple}.

    \section{Proof of main theorem}\label{generallabel}

    Let $\mathcal A$ be a finite-dimensional $C^*$-algebra over the field $\mathbb F\in\{\mathbb R, \mathbb C\}$. We will denote the matrix block decomposition of a complex $C^*$-algebra $\mathcal A$ by
    \begin{equation}\label{eq:matrixdecompositionC}
    \mathcal M_{n_1}(\mathbb C)\oplus\dots\oplus\mathcal M_{n_\ell}(\mathbb C)
    \end{equation}
    for some positive integers $n_1, \dots, n_\ell$. Similarly, the matrix block decomposition of a real $C^*$-algebra $\mathcal A$ will be denoted by
    \begin{equation}\label{eq:matrixdecompositionR}
    \mathcal M_{n_1}(\mathbb K_1)\oplus\dots\oplus\mathcal M_{n_\ell}(\mathbb K_\ell)
    \end{equation}
     where $\mathbb K_i\in\{\mathbb R, \mathbb C, \mathbb H\}$       (every~$\A$ has such decomposition, see a book by Goodearl~\cite{goodearl} for more information). If needed we will regard ${\mathcal M}_n(\KK)$ as the set of matrices corresponding to $\KK$-linear operators acting on a right $\KK$-vector space of column vectors $\KK^n$ and with respect to the standard basis $e_1,\dots,e_n$. The standard  pairing of vectors  $x,y\in\KK^n$ will be denoted by
    $$\langle x,y\rangle_{\KK}:= y^\ast x,$$
    where a row vector $y^\ast$ is a conjugated transpose of $y$. Note that $\KK^n$ is an Euclidean space under the $\FF$-linear inner product
    $$\langle x,y\rangle_{\FF}:=\begin{cases} \mathop{\mathrm{Re }} y^*x;&\FF=\RR\\y^*x;&\FF=\CC\end{cases}.$$
                    For a vector $v$ in a normed space~$V$, $$v^\bot=\{u;\; v\perp u\}\quad\text{ and }\quad {}^\bot v=\{u;\; u\perp v\},$$ are the \emph{outgoing} and \emph{incoming neighbourhood} of $v$, respectively.  Also, $v$ is  \textit{left-symmetric}  if $v\perp x$ implies $x\perp v$, or equivalently, if  $v^\bot\subseteq {}^\bot v$. For $\mathcal S\subseteq V$, let
    $$\mathcal L_{\mathcal S}:=\{v\in\mathcal S;\;\;v^\bot	\cap {\mathcal S}\subseteq {}^\bot v\cap {\mathcal S}\}.$$
    denote the set of all left-symmetric vectors relative to $\mathcal S$.
        In particular, if $\mathcal S=V$, then $\mathcal L_{V}$ is the set of all left-symmetric vectors. 
        It follows from \cite[Theorem 1.2]{abelian} that
        $\mathcal L_{\mathcal A}^\bot:=\bigcap_{v\in \mathcal L_{\mathcal A}}v^\bot$ (the \textit{nonpseudo-abelian summand} of a $C^*$-algebra $\A$) is the sum of $n_i$-by-$n_i$ blocks with $n_i>1$,  while $\mathcal L_{\mathcal A}^{\bot\bot}:=\bigcap_{v\in \mathcal L_{\mathcal A}^\bot} v^\bot$
        (the \textit{pseudo-abelian summand} of $\A$) is the sum of all $1$-by-$1$ blocks in \eqref{eq:matrixdecompositionC}--\eqref{eq:matrixdecompositionR}. 
      Observe that $\mathcal L_{\mathcal A}^{\bot\bot}$
        is  a  $C^\ast$-algebra; abelian if and only if there are no quaternionic $1$-by-$1$ blocks. We say $\A$ is \textit{pseudo-abelian} if $\A=\mathcal L_{\mathcal A}^{\bot\bot}$; and \textit{not pseudo-abelian} if $\A\neq \mathcal L_{\mathcal A}^{\bot\bot}$.
       
       For $A=\bigoplus_{k=1}^\ell A_k\in\A$, we define $$M_0(A) =\{x \in \oplus_{k=1}^\ell\mathbb K_k^{n_k} ;\;\; \|Ax\|=\|A\|\cdot\|x\|\}.$$ Note that $M_0(A_k):=M_0(A)\cap(0\oplus \mathbb K_k^{n_k}\oplus 0)$ is a $\mathbb K_k$ vector subspace, see \cite[Lemma 3.1]{simple}. We will rely on the following Stampfli-Magajna-Bhatia-\v Semrl classification of BJ orthogonality in $C^*$ norm.
    
        \begin{proposition}[\cite{abelian}, Lemma 2.1]\label{prop:M_0(A)}
        Let $\A=\bigoplus_{k=1}^\ell\mathcal M_{n_k}(\mathbb K_k)$ be a $C^*$-algebra over~$\FF$ and  $A,B\in\A$. Then, $B\in A^\bot$ if and only if $\langle Ax,Bx\rangle_{\FF} =0$ for some  normalized  $x\in M_0(A)$. Moreover,  if  $\|A_i\|<\max\{\|A_k\|;\;\;1\leq k\leq\ell\}=\|A\|$, then $$A^\bot = \big(A_1\oplus\dots\oplus A_{i-1}\oplus 0_{n_i}\oplus A_{i+1}\oplus\dots\oplus A_\ell\big)^\bot.$$
    \end{proposition}

    We  call $A\in\A$ to be \textit{smooth} if  there does not exist  $B\in\A$ such that $B^\bot\subsetneq A^\bot$.
    It was proved in~\cite[Lemmas 2.6 and  2.5]{abelian} that $A=\bigoplus_{k=1}^\ell A_k\in\A$ is smooth  if and only if
    there exists exactly one index $j$ such that $\|A_j\|=\|A\|$ and $\mathrm{dim}_{\mathbb K_j}M_0(A_j)=1$. Moreover, if $A$ is smooth and $x=\bigoplus_{k=1}^\ell x_k\in M_0(A)$, then $x_k=0$ for all $k\neq j$~and $$A^\bot=\big(0\oplus ((A_jx_j)x_j^*)\oplus 0\big)^\bot.$$
    By~\cite[Lemma 2.5]{abelian} this definition also agrees with the classical definition of smoothness, i.e., of having a unique supporting functional. For example, $A=0\oplus E_{st}^j\oplus 0$ are smooth elements for all matrix units $E_{st}^j\in\mathcal M_{n_j}(\mathbb K_j)$.

    Let us  briefly discuss the key idea of the proof of Theorem~\ref{maintheorem1}. In \cite{abelian}, the finite-dimensional pseudo-abelian $C^*$-algebras were classified, by finding the  minimal number of smooth points $A_1,\dots, A_s$ required so that $\aleph:=|\{B^\bot;\; B\in\mathcal L_{\bigcap_{k=1}^sA_k^\bot}\}|<\infty$. In that case, $\aleph$ equals the number of matrix blocks and $\dim_{\FF}\A=s+\aleph$. With a few additional lemmas the underlying~field and the number of real, complex, and quaternionic blocks were also found. We will follow the same~footsteps and focus on the minimal number of smooth points $A_1,\dots, A_s$ in~$\mathcal L_\A^\bot$ which are required so that $\mathcal L_{\bigcap_{i=1}^s A_i^\bot}\neq \{0\}$. Let us start by proving the existence of such a finite tuple.

    \begin{lemma}\label{coordinatewisesymmetricity}
    Let $\mathcal A=\bigoplus_{i=1}^\ell\mathcal M_{n_i}(\KK_i)$ be a finite-dimensional $C^*$-algebra over~$\mathbb F$ with $n_i\geq 2$. Then, there exists finitely many smooth elements $A_1,\dots,A_t$ of
    $\mathcal A$ such that the set $\mathcal B=\bigcap_{i=1}^t A_i^\bot$ contains a non-zero left-symmetric element.
    Moreover, for every  (possibly infinite)  collection of smooth points $(A_\alpha)_\alpha$  there exist $\mathbb F$-subspaces $\mathcal S_i\subseteq\mathcal M_{n_i}(\KK_i)$ such that $\bigcap_{\alpha} A_\alpha^\bot=\bigoplus_{i=1}^\ell \mathcal S_i$, and for any left-symmetric $Z=\bigoplus_{i=1}^\ell Z_i\in \mathcal L_{\bigcap_{\alpha}A_\alpha^\bot}=\mathcal L_{\bigoplus_{i=1}^\ell \mathcal S_i}$, we have  $Z_i\in \mathcal L_{\mathcal S_i}$ for all $1\leq i\leq \ell$.
    \end{lemma}
    \begin{proof} To prove the existence, we note that the standard matrix units $E_{st}^1\oplus 0$ are smooth elements and, using Proposition \ref{prop:M_0(A)}, we have $$\bigcap\{(\mu E_{st}^1\oplus 0)^\bot;\;\;s\neq t \text{ and } \mu\in\{\boldsymbol{1,i,j,k}\}\cap\mathbb K_1\}=\KK_1^{n_1}
         \oplus\big(\bigoplus_{i=2}^\ell\mathcal M_{n_i}(\KK_i)\big).$$ The set on the right  is a $C^*$ algebra that contains pseudo-abelian direct summand and hence it does contain a non-zero left-symmetric point by \cite[Lemma 2.2]{abelian}.

    Now, we fix arbitrary collection of smooth elements $(A_\alpha)_\alpha$ of
    $\mathcal A$ and let $\mathcal B=\bigcap_\alpha A_\alpha^\bot$.       By definition of smoothness there exists a unique $j$ and a normalized  vector $u\in \mathbb K_j^{n_j}$ such that $M_0(A_\alpha)=0\oplus u{\mathbb K_j}\oplus 0$. So, if we define $\mathbb F$-linear functional $f_\alpha\colon \A\to\FF$ by $$f_\alpha\big(\oplus_{k=1}^\ell C_k\big)=\langle C_ju, (A_\alpha)_j u \rangle_{\FF},$$ where $(A_\alpha)_j$ is the $j$-th summand of $A_\alpha$, then by Proposition \ref{prop:M_0(A)}, we have $A_\alpha^\bot=\mathrm{ker}f_\alpha$. Hence, $\bigcap_\alpha A_\alpha^\bot=\bigcap_\alpha \mathrm{ker}f_\alpha$. Now, $\bigcap_\alpha\mathrm{ker}f_\alpha$ is a subspace of $\mathcal A$, equal to $\bigoplus_{i=1}^\ell \mathcal S_i$ for $\FF$-linear subspace $\mathcal S_i$ of $\mathcal A$. This proves the second part of the claim.

    To prove the last part, let $B=\bigoplus_{i=1}^\ell B_i\in \mathcal B=\bigoplus_{i=1}^\ell \mathcal S_i$ be a left-symmetric element. We need to prove $B_i\in \mathcal L_{\mathcal S_i}$. Without loss of generality, $i=1$. If $B_1=0$, then clearly
         $B_1\in \mathcal L_{\mathcal S_1}$. Let $B_1\neq0$ and
           assume erroneously        it is not left-symmetric in $\mathcal S_1$, i.e., there is $C_1\in \mathcal S_1$ such that $B_1\perp C_1$ but $C_1\not\perp B_1$. By Proposition \ref{prop:M_0(A)}, $B_1\perp C_1$ implies there exists a normalized  $x_1\in M_0(B_1)\subseteq \mathbb K_1^{n_1}$ such that \begin{equation}\label{eq76543}\langle B_1x_1, C_1x_1\rangle_\FF=0.\end{equation}
         Then $C_1\not\perp B_1$ implies that $C_1\neq 0$ and for all normalized $y_1\in M_0(C_1)$, we have \begin{equation*}\langle C_1y_1, B_1y_1\rangle_\FF\neq 0.\end{equation*}

    Consider $C=C_1\oplus 0\in \mathcal B$ and note that $M_0(C)=M_0(C_1)\oplus 0$. Hence, $C\not\perp B$.
         Let us prove $B\perp C$ (this will then contradict the fact that $B\in \mathcal L_{\mathcal B}$, and finish the proof). There are two cases to consider.

    \textbf{Case 1:}  $\|B_1\|=\|B\|$. Let $x=x_1\oplus 0$, where $x_1\in M_0(B_1)$ satisfies \eqref{eq76543}.
    By Proposition \ref{prop:M_0(A)}, $M_0(B_1)\oplus 0\subseteq M_0(B)$ so  $x\in M_0(B)$. Then, $\sum\limits_{i=1}^\ell \langle B_ix_i, C_ix_i\rangle_\FF=\langle B_1x_1, C_1x_1\rangle_\FF=0,$ and hence $B\perp C$.

    \textbf{Case 2:}  $\|B_1\|<\|B\|$. Then $M_0(B)\subseteq 0\oplus\big(\bigoplus_{i=2}^\ell\mathbb M_{n_i}(\mathbb K_i)\big)$, so for any $x\in M_0(B)$, we have
         $\sum\limits_{i=1}^\ell \langle B_ix_i, C_ix_i\rangle = 0+\sum\limits_{i=2}^\ell \langle B_ix_i, 0_{n_i}x_i\rangle=0.$ \end{proof}

    \begin{lemma}\label{minimalsmooth}
    Suppose $\A=\bigoplus_{k=1}^\ell\mathcal M_{n_k}(\KK_k)$ with $n_k\geq 2$ and let $A_1,\dots ,A_s\in\A$ be smooth elements where $s$ is the minimal number required to ensure that $\bigcap_{i=1}^s A_i^\bot$ contains a non-zero left-symmetric element. Then, $s\geq 1$ and there exists $j$ with \begin{equation}
        M_0(A_i)\subseteq 0\oplus\mathbb K_j^{n_j}\oplus 0\text{ for all } 1\leq i\leq s.
    \label{eq9090}\end{equation}
              \end{lemma}
    \begin{proof} By \cite[Lemma 2.2]{abelian}, we have $s\geq 1$.
        By Lemma \ref{coordinatewisesymmetricity}, for any non-zero left-symmetric  $Z=\bigoplus_{i=1}^\ell Z_i\in \bigcap_{i=1}^s A_i^\bot=\bigoplus_{i=1}^\ell \mathcal S_i$, we have $Z_i\in \mathcal L_{\mathcal S_i}$. Let $j_0$ be fixed such that $\|Z_{j_0}\|=\|Z\|$.
                  We claim that  $Z_i=0$ for each $i\neq j_0$.  Assume otherwise.          Then $Z\perp \big(0\oplus Z_i\oplus 0\big)\in \bigcap_{i=1}^s A_i^\bot$ but these are not mutually orthogonal which contradicts the left-symmetricity of
                 $Z\in \bigcap_{i=1}^s A_i^\bot$. Now, let $A_{i_1},\dots,A_{i_m}$ be those among $A_1,\dots,A_s$ such that \begin{equation}\label{5898}M_0(A_{i_p})\subseteq 0\oplus\mathbb K_{j_0}^{n_{j_0}}\oplus0\text{ for all }1\leq p\leq m.\end{equation}
                   It easily follows (c.f. the proof of Lemma~\ref{coordinatewisesymmetricity}) that
         $$\bigcap_{p=1}^mA_{i_p}^\bot=\mathcal M_{n_1}(\KK_1)\oplus\dots\oplus \mathcal M_{n_{j_0-1}}(\mathbb K_{j_0-1})\oplus \mathcal S_{j_0} \oplus \mathcal M_{n_{j_0+1}}(\mathbb K_{j_0+1})\oplus\dots\oplus \mathcal M_{n_\ell}(\KK_\ell).$$
                   Since $Z_i\in{\mathcal L}_{\mathcal S_i} $, we easily see that  $Z=0\oplus Z_{j_0}\oplus 0$ is left-symmetric already in $ \bigcap_{p=1}^mA_{i_p}^\bot$.
                   By the minimality of $s$, we have $m=s$ and  \eqref{eq9090} follows from~\eqref{5898}.
    \end{proof}

    The next two technical lemmas will be needed in the proof of our main procedure to extract individual blocks.

    \begin{lemma}\label{technical2}
       Let ${\mathcal V}:=\left(
    \begin{smallmatrix}
     0 & \KK \\
     \KK & \KK \\
    \end{smallmatrix}
    \right)\subset{\mathcal M}_2(\KK)$. Then  $\mathcal L_{\mathcal V}\neq \{0\}$. Moreover,
    \begin{equation}\label{999009} \mathcal V\cap \bigcap_{Z\in\mathcal L_{\mathcal V}}Z^\bot=\{0\}.
    \end{equation}
    \end{lemma}

    \begin{proof}

    We consider two  cases.

    \textbf{Case 1:} $\KK=\FF\in\{\RR,\CC\}$. We first prove  $A=E_{22}$ is left-symmetric. Clearly,
         $$A^\bot\cap {\mathcal V}=\left(
    \begin{smallmatrix}
         0 & \KK \\
     \KK & 0 \\
    \end{smallmatrix}\right).$$ Any matrix $B=\left(
    \begin{smallmatrix}
     0 & c \\
     d & 0 \\
    \end{smallmatrix}\right)\in A^\bot\cap {\mathcal V}$ attains its norm on $x=e_1$ or $x=e_2$ (the standard basis of  $\KK^2$). In both cases, we have $\langle Bx,Ax\rangle_{\FF}=0$ giving the desired  $B\perp A=E_{22}$.

    Next,  we prove $A=E_{12}$ is left symmetric. Here, $A^\bot\cap {\mathcal V}$
    consists of matrices of type $B=\left(
    \begin{smallmatrix}
     0 & 0 \\
     d & c \\
    \end{smallmatrix}\right)$. Then,  $\mathrm{Im}\,B\subseteq e_2 \KK$ while $\mathrm{Im}\,A=e_1 \KK$.
              So, $B\perp A$, as claimed. Finally, $A=E_{21}$ is also left-symmetric because  the adjoint operation which is a conjugate-linear isometry and hence   a BJ-isomorphism, swaps $E_{12}$ and $E_{21}$ and leaves the $\FF$-linear subspace ${\mathcal V}$ invariant. This proves \eqref{999009} because we have
    $$\mathcal V\cap E_{22}^\bot\cap E_{12}^\bot\cap E_{21}^\bot=\{0\}.$$

    \textbf{Case 2:} $\FF=\RR$ and $\KK=\CC$ or $\HH$. We   first prove that $A=E_{22}$ is left-symmetric.  Its  outgoing neighbourhood, relative to ${\mathcal V}$, consists of matrices of the~form
                                                 $$
    B=
    \begin{pmatrix}
     0 & a \\
     b &  {\bf q} \tau\\
    \end{pmatrix}\in\begin{pmatrix}
     0 & \KK \\
     \KK & \boldsymbol{q}\RR \\
    \end{pmatrix},$$
     where  $\boldsymbol{q}\in\KK$ is a   purely imaginary  unimodular number and $\tau\in\RR$. By homogeneity of BJ orthogonality we can assume $\tau=1$ or $\tau=0$.
                        If $\tau=1$, then,  in computing the norm of
    $B$,  one first writes
    $B^\ast B=\left(  \begin{smallmatrix}
     0 & \overline{b} \\
     \overline{a} & \overline{\bf q} \\
    \end{smallmatrix}\right).    \left(\begin{smallmatrix}     0 & a \\
     b & {\bf q} \\
    \end{smallmatrix}\right)=\left(\begin{smallmatrix}
     | b| ^2 & \bar{b}{\bf q} \\
     \bar{\bf q}b & | a| ^2+1 \\
    \end{smallmatrix}\right)$ and then embeds its coefficients (notice that only $z=\overline{b}{\bf q}$ and its conjugate can be nonreal) isometrically into $\CC$ where singular value decomposition is easily available. It yields $\|B\|^2=\frac{\sqrt{\left(| a| ^2+| b|^2+1\right)^2 -4|a|^2|b|^2}+| a|^2+| b| ^2+1}{2}$
       and $B$
              achieves its norm on
    $x=\left(\begin{smallmatrix} \boldsymbol{q}\lambda\\ \overline{\boldsymbol{q}}b\boldsymbol{q}\end{smallmatrix}\right)$, with $\lambda:=(\|B\|^2-1-|a|^2)$.
                 One  can easily see that $$\langle Bx,Ax\rangle_{\FF}        = \mathrm{Re}((Ax)^\ast(Bx))=\mathrm{Re}( (1+\lambda)  \overline{\overline{\boldsymbol{q}}b\boldsymbol{q}}\cdot b\boldsymbol{q} )=
       (1+\lambda)\,\mathrm{Re}(\overline{b\boldsymbol{q}}\cdot\boldsymbol{q}\cdot b\boldsymbol{q} )=0$$ because conjugation of a purely imaginary $\boldsymbol{q}\in\KK$ is again purely imaginary.
                    So $B\perp A$.
        If $\tau=0$ we get the second typical class of matrices in $A^\bot\cap {\mathcal V}$, namely $B=\left(
    \begin{smallmatrix}
     0 & a \\
     b & 0 \\
    \end{smallmatrix}\right)\in A^\bot$. This is treated as in  Case 1 and again gives $B\perp A$.

    We now prove that $E_{12}$ is also left-symmetric. In this case, the first class of typical matrices in $A^\bot\cap{\mathcal V}$ take the form   $B=\left(\begin{smallmatrix}0 & \boldsymbol{q} \\b & a \\ \end{smallmatrix}\right)$ for some purely imaginary unimodular number $\boldsymbol{q}$. Following the previous footsteps  we easily compute  that $$\|B\|^2=\frac{1}{2} \left(\sqrt{(| a|^2+|b|^2+1)^2-4|b|^2}+|a|^2+|b|^2+1\right)$$
    and that $B$ achieves its norm on $x=\left(\begin{smallmatrix} a\lambda\\ \overline{a}ba\end{smallmatrix}\right)$, with $\lambda:=(\|B\|^2-1-|a|^2)$.
                   Clearly,      $Bx=\left(\begin{smallmatrix}
    \boldsymbol{q} \bar{a}ba\\
    \ast\end{smallmatrix}\right)$
          and hence $\langle Bx, Ax\rangle_{\mathbb F}=\mathrm{Re}(\overline{\overline{a}ba}\cdot \boldsymbol{q}\bar{a}ba)=0$.
                The second class of  matrices from $A^\perp\cap{\mathcal V}$ take the   form $B=\left(
    \begin{smallmatrix}
     0 & 0 \\
     b & a \\
    \end{smallmatrix}\right)= e_2(e_1 \overline{b}+e_2 \overline{a})^\ast$,  attain their norm  on    $x=(e_1 \overline{b}+e_2 \overline{a})$, and map it into $Bx=e_2$. Since $Ax=E_{12}x=e_1\overline{a}$ we get~$B\perp A$.

    Finally, the left symmetricity of $Z\in\{\boldsymbol{q}E_{22}, \boldsymbol{q}E_{12},  \boldsymbol{q}E_{21}\}$ for arbitrary unimodular purely imaginary number $\boldsymbol{q}$, follows as in Case 1 by using   BJ-isomorphisms induced by isometries $X\mapsto UXV^\ast$ or $X\mapsto X^\ast$. Hence, \eqref{999009} holds.                \end{proof}

    \begin{lemma}\label{technical1} Let $\mathcal V\subseteq\mathbb K$ be a proper $\mathbb F$-subspace. Then,
    the subspace ${\mathcal T}:=\mathcal V E_{11}+\sum_{(i,j)\neq (1,1)} (\KK E_{ij})$ of $\mathcal M_n(\mathbb K)$ contains a non-zero left-symmetric element if and only if $n=2$ and ${\mathcal V}=0$.
         \end{lemma}

    \begin{proof} The `only if' part follows from Lemma \ref{technical2}. We prove the `if' part. Let $A\in {\mathcal T}$ be left-symmetric and $\|A\|=1$. We divide the proof into several cases.

        {\bf Case 1}.  $\dim_{\KK}M_0(A)\ge 2$. Then there exists a normalized  $x\in M_0(A)$ with
        $\langle e_1,x\rangle_{\KK}:=x^\ast e_1=0$ and there exists a normalized $w\in M_0(A)$ which is $\KK$-orthogonal to $x$.  Let $B=Axx^\ast$; then  $e_1^\ast B e_1=(e_1^\ast Ax)(x^\ast e_1)=(e_1^\ast Ax)\cdot0=0$, so
                 $B\in{\mathcal T} $. Clearly, $Bw=0$,
                 hence $\langle Aw,Bw\rangle_{\FF}=0$, so $A\perp B$.  Being left-symmetric, this implies $B\perp A$. However $B$ achieves its norm only on $\KK$-multiples of $x$, that is, on $z=x\mu$ (with $\mu\in\KK$ non-zero to avoid trivialities), so, due to $Bx=Ax$,   $$0=\langle Bz,Az\rangle_{\FF}=\langle B(x\mu),A(x\mu)\rangle_{\FF}=\langle Ax\mu,Ax\mu\rangle_{\FF}=\|Ax\mu\|^2>0,$$ a contradiction. So, Case 1 is impossible.

    {\bf Case 2}. $\dim_{\KK} M_0(A)=1$. Then, if
    $n\ge 3$, the $\KK$-orthogonal complement of $M_0(A)$  has dimension at least $2$ and hence contains a normalized vector $x$ which is $\KK$-orthogonal to $e_1$.  Clearly $x$ is also $\KK$-orthogonal to every vector from $M_0(A)$, so  $B=yx^\ast\in {\mathcal T}$ and  $A\perp B$ regardless of the choice of $y$. If $Ax\neq 0$, then, by choosing $y:=Ax$ we get a contradiction as before. Therefore, the (column) rank of $A$ is at most two --- it does not annihilate a vector from $M_0(A)$ and perhaps a vector from $$M_0(A)^\bot\ominus \bigl((\KK e_1)^\bot\cap M_0(A)^\bot\bigr).$$ Note for $n=2$, the (column) rank of $A$ is always at most two.

    {\bf Subcase (a)}. $A$ has (column) rank-1. So, $A=xy^*$ for normalized  $x=\sum e_i\xi_i$ 
         and $y=\sum e_i\upsilon_i$.      First, let       $\mathcal V\neq 0$.                There exist  a unitary~$U$ which fixes $e_1$ and maps $x- e_1\xi_1=\sum_2^n  e_i\xi_i$      into $ e_2(\mathop{\mathrm{sgn}}_0(\xi_1) s\nu)$ where    $\mathrm{sgn}_0(\alpha)=\frac{\alpha}{|\alpha|}$ if $\alpha\neq 0$ while  $\mathrm{sgn}_0(0)=1$, and  $s:=\sum_2^n|\xi_i|^2=\sqrt{1-|\xi_1|^2}$ and where $\nu=1\in\HH$ if $\xi_1\neq0$ while $\nu$ is unimodular and satisfies $\nu\mathop{\mathrm{sgn}}_0(\overline{\upsilon_1})\in{\mathcal V}$ if $\xi_1=0$  (we also find a unitary $V$ which fixes $e_1$ and maps $y-e_1\upsilon_1=\sum_2^n  e_i\upsilon_i$ into $e_2(\mathop{\mathrm{sgn}}_0(\upsilon_1)c\epsilon)$; here $ c=\sqrt{1-|\upsilon_1|^2}$ while  $\epsilon=1$ if $\upsilon_1\neq0$ and $\mathop{\mathrm{sgn}}_0(\xi_1)\overline{\epsilon}\in{\mathcal V}$ if $\upsilon_1=0$, respectively),
         such that
    $$A=U^\ast ( e_1\xi_1+e_2\,\mathrm{sgn}_0(\xi_1) s\nu)(e_1\upsilon_1+  e_2\,\mathrm{sgn}_0(\upsilon_1) c\epsilon)^\ast V.$$
              Since $U {\mathcal T}V^\ast={\mathcal T}\ni A$, so $U AV^\ast\in (\mathcal M_2(\RR)\overline{\mu}\oplus 0)\subseteq \mathcal T$, where $\overline{\mu}=\mathrm{sgn}_0(\xi_1)\,\mathrm{sgn}_0(\overline{\upsilon_1})$ if $\xi_1\upsilon_1\neq0$ and $\overline{\mu}=\nu\mathop{\mathrm{sgn}}_0(\overline{\upsilon_1})\in{\mathcal V}$ or $\overline{\mu}=\mathrm{sgn}_0({\xi_1})\overline{\epsilon}\in{\mathcal V}$, if $\xi_1=0$ or $\upsilon_1=0$, respectively. Then by \cite[Lemma 2.2]{abelian},  for  $\hat{A}:=U AV^\ast\mu\in\mathcal M_2(\RR)\oplus 0$  of (column) rank-$1$, there exists $\hat{B}\in M_2(\RR)\oplus 0$ with $\hat{A}\perp \hat{B}$ but $\hat{B}\not\perp \hat{A}$. Then, $B:=\hat{B}\overline{\mu}\in \mathcal M_2(\RR)\overline{\mu}\oplus 0\subseteq \mathcal T$
         is the matrix with $A\perp B$ and $B\not\perp A$, a contradiction to left-symmetricity of $A$.

    Second let $n\geq 3$ and ${\mathcal V}=0$. Then $\xi_1=0$ or $\upsilon_1=0$. By adjoint $X\mapsto X^\ast$, which is a conjugate-linear isometry (hence a BJ-isomorphism), we can assume $\upsilon_1=0$. Now, there exists unitary $U, V$ which fixes $e_1$ such that $Ux=e_1\xi_1+e_2\mu$ and $Vy=e_2$. So, without of loss of generality, $A= (e_1\xi_1+e_2\mu)e_2^*=E_{12}\xi_1+ E_{22}\mu$.   Consider temporarily  $A$ as belonging to $\mathcal M_2(\mathbb K)$, then by \cite[Lemma 2.2]{abelian}, there exists $B\in\mathcal M_2(\mathbb K)$ such that $A\perp B$ but $B\not\perp A$. Now, we consider the map $\phi : \mathcal M_2(\mathbb K)\rightarrow \mathcal T$ defined as $\phi(X)=(X\oplus 0_{n-2})(V\oplus I_{n-3})$, where $V=\left(\begin{smallmatrix} 0&0&1\\0&1&0\\1&0&0 \end{smallmatrix}\right)$ is a permutation matrix. This map is a linear  isometric embedding. So, $\phi$ satisfies $C\perp D$ if and only if $\phi(C)\perp \phi(D)$. It implies $A=\phi(A)\perp \phi(B)$ but $\phi(B)\not\perp A$, a contradiction to left-symmetricity of $A\in\mathcal T$.

    {\bf Subcase (b)} $A$ has (column) rank-2.      Let $A=x_1y_1^\ast +\sigma x_2y_2^\ast$ be its singular value decomposition (see~\cite{zhang} if $\KK=\HH$) with $0<\sigma\leq 1$
         where $x_1=\sum\xi_i e_i$ and $x_2=\sum \zeta_i e_i $       are $\KK$-orthogonal and likewise for  $y_1=\sum \upsilon_ie_i$  and $y_2=\sum\omega_i e_i$.  If $n\ge 3$, then using $X\mapsto UXV^\ast$ for suitable unitaries $U,V$ which fix $e_1$  we can achieve that $$A=(e_1\xi_1+e_2 \mathrm{sgn}_0(\xi_1)s)( e_1\upsilon_1+e_2\mathrm{sgn}_0(\upsilon_1)c)^\ast +\sigma ( e_1\zeta_1+ e_2\beta+e_3\gamma)(e_1\omega_1+ e_2\nu+e_3\tau)^\ast$$
    for real $s=\sqrt{1-|\xi_1|^2}$ and $c=\sqrt{1-|\upsilon_1|^2}$ and suitable $\beta,\gamma,\nu,\tau\in\KK$. If $\tau\neq0$, then
    $B=( e_1\zeta_1+ e_2\beta+e_3\gamma)\overline{\tau}e_3^\ast$ has zero entry at $(1,1)$ position, so belongs to~${\mathcal T}$, and $A\perp B$ (because $A$ attains its norm on  $( e_1\upsilon_1+e_2\mathrm{sgn}_0(\upsilon_1)c)$ which is
    $\KK$-orthogonal to $e_3$) but $B\not\perp A$ because $B$ attains its norm only on multiples of  $e_3$ and   $\langle Be_3,Ae_3\rangle_{\FF}=
         \sigma|\tau|^2(|\zeta_1|^2+|\beta|^2+|\gamma|^2)\neq0$, so again $A$ is not left-symmetric.

    If $\tau=0$ but $\gamma\neq0$, then consider
    $$B=\gamma e_3(e_1\omega_1+e_2\nu+e_3\tau)^\ast$$
    which again has $(1,1)$ entry equal to $0$. Its image is spanned by $e_3$, while the image of $A$ under its norm-attaining vector $( e_1\upsilon_1+e_2\mathrm{sgn}_0(\upsilon_1)c)$ is spanned by $e_1,e_2$. Hence, $A\perp B$. That $B\not\perp A$ is similar as in the previous case.

    The only  case left is when, after modifying $A$ with $X\mapsto UXV^\ast$ for unitaries  $U,V$ which fix $e_1$ (and absorbing the appropriate scalar in the second factor),
         we~have
         $$A=(c_1 e_1+s_1e_2)(c_2 e_1\mu+s_2e_2\mu)^\ast+\sigma (-s_1 e_1+c_1e_2)(-s_2 e_1\nu+c_2e_2\nu)^\ast\in \bM_2(\KK)\oplus 0$$ for some unimodular $\mu, \nu\in\KK$, $0<\sigma\le 1$ and  real $c_i,s_i$ with $|c_i|^2+|s_i|^2=1$ for $i=1,2$ (note here $c_1c_2\overline{\mu} +\sigma s_1s_2\overline{\nu}=A_{11}\in {\mathcal V}$). Then,
    $B=(e_2\mathrm{sgn}_0(c_1)\overline{\nu}) (-s_2 e_1+c_2e_2)^\ast$ has zero $(1,1)$ entry, so $B\in\mathcal T$. Also,  $\langle Ax, Bx\rangle_{\mathbb F}=0$ for  $x=(c_2e_1+s_2e_2)\in M_0(A)$ so $A\perp B$, but $B$ achieves the norm only on multiples of  $(-s_2 e_1+c_2e_2)$ and maps it into $\mathrm{sgn}_0(c_1)e_2\overline{\nu}$, while $A$ maps  it into $\sigma(-s_1e_1+c_1e_2)\overline{\nu}$. So, unless $c_1=0$ we have $B\not\perp A$, again contradicting left-symmetricity of $A$. So, we must have
    $$A=e_2(c_2 e_1\mu+s_2e_2\mu)^\ast+\sigma (- e_1)(-s_2 e_1\nu+c_2e_2\nu)^\ast.$$
    Then, $(1,1)$ entry of $A$ is  $A_{11}=\sigma s_2\overline{\nu}\in\mathcal V$. Now,  $B:=\sigma(-e_1)(-s_2 e_1\nu+c_2e_2\nu)^\ast$ has the same $(1,1)$ entry as $A$, so $B\in\mathcal T$, and it attains its norm at scalar multiple of $-s_2 e_1\nu+c_2e_2\nu$ while $\langle A(-s_2 e_1\nu+c_2e_2\nu), B(-s_2 e_1\nu+c_2e_2\nu)\rangle_{\FF}=\sigma^2\neq 0$. So  $B\not\perp A$. On the other hand, we have $A\perp B$ because $A$ attains norm on $c_2 e_1\mu+s_2e_2\mu$, while $B(c_2 e_1\mu+s_2e_2\mu)=0$.
         Again, $\mathcal T$ does not contain a non-zero left-symmetric element.

    So, in all the  (sub)cases, we get that a non-zero left-symmetric element is possible in $\mathcal T$ only in the case when $n=2$ and $\mathcal V=0$.      
    \end{proof}

    Notice for a $C^*$-algebra $\A=\bigoplus_{i=1}^\ell\mathcal M_{n_i}(\KK_i)$, that  $(0\oplus \mathcal M_{n_j}(\mathbb K_j)\oplus 0)^\bot$ coincides with  all the elements of $\mathcal A$ which vanish at $j$-summand.

    \begin{corollary}\label{MnF} Let $\A=\mathcal M_{n_1}(\mathbb K_1)\oplus\dots\oplus\mathcal M_{n_\ell}(\mathbb K_\ell)$ with $n_1,\dots, n_\ell\ge 2$. Assume there exists a smooth element $A\in\A$ such that $A^\bot$ contains a non-zero left-symmetric point, i.e., $\mathcal L_{A^\bot}\neq \{0\}$. Then there exists $j$ with $n_j=2$, $\KK_j=\mathbb F$ such that $M_0(A)\subseteq 0\oplus\KK_j^{n_j}\oplus0=0\oplus\mathbb F^2\oplus 0.$ Moreover, in this case \begin{equation*}A^\bot\cap\bigg(\bigcap_{Z\in\mathcal L_{A^\bot}}Z^\bot\bigg)=(0\oplus\mathcal M_{n_j}(\KK_j)\oplus 0)^\bot.
         \end{equation*}
    \end{corollary}
    \begin{proof} Since $A$ is a smooth element, there exists $j$ such that $M_0(A)\subseteq
         0\oplus \KK_j^{n_j} \oplus 0$ and $\dim_{\KK_j}M_0(A)=1$.  Hence, there exists $ y_j\in\KK_j^{n_j}$ such that  $M_0(A)=0\oplus y_j\KK_j\oplus 0$ and so, if $y:=0\oplus y_j\oplus 0$, then, by the first part of Proposition \ref{prop:M_0(A)},
           $Ayy^\ast\in\A$ and $A\in\A$ share the same outgoing neighborhoods (i.e., $(Ayy^\ast)^\bot=A^\bot$).
     Hence we can assume with no loss of generality that $A=xy^\ast\in 0\oplus {\mathcal M}_{n_j}(\KK_j)\oplus 0$ for some column vectors $x,y\in\KK_j^{n_j}$ with $\|x\|=\|y\|=1$. Choose  unitary $U,V$ so that $Ux=e_1$ and $Vy=e_1$; since $X\mapsto UXV^\ast$ is isometry, hence a BJ-isomorphism, we can  assume that  $A=0\oplus(e_1e_1^\ast)\oplus 0$.

    If $n_j\geq 3$, then, by Lemma \ref{technical1}, $A^\bot$ has no non-zero left symmetric point, so $n_j\geq 3$ is not possible. Likewise if $n_j=2$ and $({\mathbb F},{\mathbb K})$ equals $({\mathbb R},{\mathbb C})$ or $({\mathbb R},{\mathbb H})$ (here, besides Lemma \ref{technical1}, we also use the fact that $0\oplus X_j\oplus 0\in\A$ is left-symmetric if and only if $X_j$ is left-symmetric in $\mathcal M_{n_j}(\mathbb K_j)$; see, e.g., Lemma \ref{coordinatewisesymmetricity}).
         So, the only possibility left is $n_j=2$ and $\mathbb K_j=\mathbb F$. For this case,  if  $Z\in A^\bot$ is left-symmetric, then, by Lemma~\ref{coordinatewisesymmetricity},
    it can be non-zero only in $j$-th summand. Hence, the proof follows by Lemma \ref{technical2}.
    \end{proof}

    Finally, we give a procedure to extract particular blocks of size $n$-by-$n$ for $n\geq 2$ in a finite-dimensional $C^*$-algebra. Let $\mathrm{sm}(\A)$ denote the set of all smooth points in $\A$.
     Given an integer ${\kappa}$ and a $C^*$-algebra $\mathcal B$, we define $$\mathcal P_{\kappa}(\mathcal B)=\{\boldsymbol{X}:=\{X_1,\dots,X_{\kappa}\}\subseteq \mathrm{sm}(\mathcal B);\;\; \mathcal L_{\boldsymbol{X}^\bot}:=\mathcal L_{\bigcap_{i=1}^{\kappa} X_i^\bot}\neq \{0\}\}.$$
         We remark that, if $\mathcal B$ is finite-dimensional, then, by Lemma \ref{coordinatewisesymmetricity}, there always exists ${\kappa}$ such that $\mathcal P_{\kappa}(\mathcal B)\neq \emptyset$.   In the lemma below, the minimal such $\kappa$ is denoted by $s$.

    \begin{lemma}\label{procedure} Let $\A$ be a finite-dimensional $C^*$-algebra with no pseudo-abelian direct summand
         and      let $A_1,\dots, A_s\in\mathrm{sm}(\A)$ be such that $\bigcap_{i=1}^s A_i^\bot$ contains a non-zero left-symmetric point with $s$ being the minimal integer to allow this.      Consider the following recursive procedure:
    \begin{align*}
    S_0&:=\{\boldsymbol{A}=\{A_1,\dots, A_s\}\},\\
    S_{m+1}&:=\big\{\boldsymbol{X}\in\mathcal P_s(\A);\;\; \exists \boldsymbol{Y}\in S_m \text{ such that }  |\boldsymbol{X}\cap \boldsymbol{Y}|\ge 1\},
    \end{align*}
         and define $\mathfrak S=\bigcap_{\boldsymbol{X}\in \bigcup_m S_m} \boldsymbol{X}^\bot$. Then, there exists index $j$ such that $M_0(A_i)\subseteq 0\oplus\KK_j^{n_j}\oplus 0$ for all $1\leq i\leq s$ and
    \begin{align}\label{eq987654}\mathfrak S\cap\bigg( \bigcap_{Z\in\mathcal L_{\mathfrak S}} Z^\bot
         \bigg)\cap \bigg(\bigcap_{Z\in\mathcal L_{\boldsymbol{A}^\bot}} Z^\bot\bigg)=(0\oplus \mathcal M_{n_j}(\mathbb K_j)\oplus 0)^\bot.
    \end{align}
                        \end{lemma}
    \begin{remark}\label{remark2.8}
    In the above procedure we have $s\ge 1$;  if $s=1$, then $S_m=S_0=\{\{A_1\}\}$ for all~$m$ and $\mathfrak{S}=A_1^\bot$. Recall  also that the right-hand side,  $(0\oplus \mathcal M_{n_j}(\mathbb K_j)\oplus 0)^\bot$ coincides with  all the elements of $\mathcal A$ which vanish at $j$-summand.
         \end{remark}
    \begin{proof} The case $s=1$ was proved in Corollary \ref{MnF}. In the sequel we assume $s\ge 2$. We note that by Lemma \ref{minimalsmooth}, there exists index $j$ such that \begin{equation}\label{recursivenorm}M_0(A_i)\subseteq 0\oplus\KK_j^{n_j}\oplus 0
         \qquad\text{ for all } 1\leq i\leq s.\end{equation} Let $X\in \boldsymbol{X}\in S_1$. So, there exists $A_i\in \boldsymbol{A}\in S_0$ and smooth elements  $X_3,\dots, X_s$ such that $X^\bot\cap A_i^\bot\cap \bigcap_{i=3}^s X_i^\bot$ contains a non-zero left-symmetric point. Now, by~\eqref{recursivenorm} and Lemma~\ref{minimalsmooth}, we get  $$M_0(X)\subseteq 0\oplus \mathbb K_j^{n_j}\oplus 0.$$
         Recursively, this holds for all $X\in\boldsymbol{X}\in S_m$, $m\in{\mathbb N}$. Consequently,  in \eqref{eq987654}, the set  $\mathfrak{S}=\bigcap_{\boldsymbol{X}\in \bigcup_m S_m} \boldsymbol{X}^\bot$  contains all elements in $\A$  which vanish at   $j$-th summand, that is, it contains   the set on the right-hand side. As such, if  $Z\in \mathfrak{S}$
    or $Z\in \boldsymbol{A}^\bot$
    is (relative) left-symmetric, then, by Lemma~\ref{coordinatewisesymmetricity},  it can be non-zero only in $j$-th summand. Then, the left-hand side of~\eqref{eq987654}  contains   the set on the right-hand side.

    To prove they are equal we only have to show that the $j$-summand of the left-hand side vanishes.  The basic idea is to start with a given $s$-tuple  $\boldsymbol{A}$ and modify it with a
    BJ-isomorphism
         that fixes one element within an $s$-tuple. Then keep repeating~it. These recursively obtained $s$-tuples will have the same properties as the original one. Moreover, the fixed member achieves its norm only in the  $j$-th block so  the whole $s$-tuple will achieve the norm only in $j$-th block. So without loss of generality, we can assume $\A=\mathbb M_{n}(\mathbb K)$ consists of a single summand (note $n\ge 2$ because $\A$ does not contain pseudo-abelian summand).
    We will use BJ-isomorphisms induced by (conjugate)linear isometries (i) $X\mapsto UXV^\ast$ for  unitary $U,V$~or~(ii)~$X\mapsto X^\ast$.

    We will also use the following two facts:
         \begin{equation*}
        xy^\ast \perp B\Longleftrightarrow \langle By,x\rangle_{\FF}=0.
             \end{equation*}
                   \begin{equation}\label{neweq}
    \langle x+ y\mu, z\rangle_\FF=0 \hbox{ for all unimodular }\mu\in\KK\Longleftrightarrow\langle x, z\rangle_\FF=0\text{ and }\langle y, z\rangle_\KK=0.\end{equation}
     The first one follows by Proposition~\ref{prop:M_0(A)}.  The second one
    is clear if $\KK=\FF\in\{\RR,\CC\}$. In the remaining cases $(\FF,\KK)\in\{(\RR,\CC),(\RR,\HH)\}$  we note $0=\langle x+y\mu , z\rangle_\FF = \langle x, z\rangle_\FF+\mathrm{Re}(z^\ast y\mu)$, so choosing $\mu$ such that $(z^\ast y)\mu$ is purely imaginary number, we get $\langle x, z\rangle_\FF=0$. Choosing next $\mu=\frac{\overline{z^\ast y}}{|z^\ast y|}$, we get $|\langle y,z\rangle_\KK|=0$, i.e. $\langle y,z\rangle_\KK=0$.

    As in the proof of Corollary \ref{MnF}, we can assume  $A_i=x_iy_i^\ast$ for some normalized  column vectors $x_i,y_i\in\KK^{n}$ and  $A_1=e_1e_1^*$.  We now consider several cases.

    \textbf{Case 1:}  $A_1,\dots,A_s\in \KK e_1e_1^\ast$. By Lemma \ref{technical1} this  is only possible for $n=2$ in which case Lemma \ref{technical2} implies  ${\mathcal V}:=\bigcap_{i=1}^k A_i^\bot=\left(\begin{smallmatrix}
        0&\mathbb K\\ \mathbb K&\mathbb K
    \end{smallmatrix}\right)$ and
    ${\mathcal V}\cap \bigcap_{Z\in{\mathcal L}_{\mathcal V}}Z^\bot=\{0\}$. This gives \eqref{eq987654}.

    \textbf{Case 2:}  $A_i\notin \KK e_1e_1^\ast$ for all $i\neq 1$.  After rearranging we can assume that $A_2=x_2y_2^\ast$ is such that at least one among $x_2,e_1$ and $y_2,e_1$ are $\KK$-linearly independent. By applying suitable unitaries which fix $e_1$ we can achieve that
    $$x_2=e_1 \alpha+e_2\beta\quad \hbox{ and }\quad y_2=e_1\gamma+e_2\delta$$  where $\alpha= e_1^\ast x_2$ and $\gamma=e_1^\ast y_2 $ while $\beta,\gamma\in\KK$ are arbitrary as long as    $|\alpha|^2+|\beta|^2=|\gamma|^2+|\delta|^2=\|x\|^2=\|y\|^2=1$. If needed we apply the conjugate-linear isometry $X\mapsto X^\ast$, which also fixes $A_1=e_1e_1^\ast$, to achieve that $\beta\neq0$. Moreover, if $n=2$ and ${\boldsymbol A}\not\subseteq \KK e_1e_1^\ast\cup \KK e_2e_2^\ast\cup \KK e_2e_1^\ast\cup \KK e_1e_2^\ast$, we can further select $A_2$ so that also $\alpha\neq0$.
     We can now start the procedure.
         There exist unitaries $U,V$, which fix $e_1$ and map $e_2\beta$ into $ e_i\beta\mu$ for any given  $i\in\{2,\dots,n\}$  and any given unimodular $\mu\in\KK$. Applying the corresponding isometries $X\mapsto UXV^\ast$ on initial $s$-tuple  $\{A_1,A_2,\dots\}\in S_0$ and considering the images of $A_2$ gives that
    \begin{equation}\label{eq:Aij(mu)}
     \bigl\{A_1\,,\,A_{ij}^{\pm}(\mu):=(e_1\alpha+e_i\beta\mu)(e_1\gamma\pm e_j\delta)^\ast\,,\dots\bigr\}\in {S}_1
     \end{equation}
     for every $i,j\ge 2$ and every unimodular $\mu\in\KK$. We  also choose a unitary $U=(e_1\alpha+e_2\beta)(e_1\alpha+e_2\beta)^\ast +(-e_1|\beta|\mu +e_2\beta'\bar{\alpha}\mu)(-e_1|\beta| +e_2\beta'\bar{\alpha})^\ast +\sum_3^n e_ie_i^\ast$ (here, $\beta':=\mathop{\mathrm{sgn}}_0(\beta)$     ); the  isometry $X\mapsto UX$  fixes $A_2$ and maps the initial $s$-tuple into
     \begin{equation}\label{eq:A1(mu)}
         \bigl\{A_1(\mu)=(e_1(| \alpha |^2+\mu  | \beta |^2)+e_2\beta  \bar{\alpha}(1-\mu))e_1^\ast\,,\, A_2,\dots\bigr\}\in {S}_1.
     \end{equation}
    Likewise there exists an isometry $X\mapsto XV_{\pm}^\ast$
         which fixes $A_{2j}^{\pm}(\mu)\in S_1$ and maps the $s$-tuple from~\eqref{eq:Aij(mu)}  into
    $$\bigr\{A_1'(\mu)=e_1(e_1(| \gamma|^2+\overline{\mu}  | \delta |^2)\pm e_j(1-\overline{\mu}){\gamma  \bar{\delta}})^\ast\,,\, A_{2j}^{\pm}(\mu),\dots\}\in S_2$$
    In particular, if $X\in (S_0\cup S_1\cup S_2)^\bot$, then $A_{ij}^{\pm}(\mu) \perp X$ so that
    $\langle e_1\alpha+e_i\beta\mu,X(e_1\gamma\pm e_j\delta)\rangle_{\FF}=0$ for every unimodular $\mu$. Since $\beta\neq0$ we get from ~\eqref{neweq} that
    $e_i^\ast X(e_1\gamma\pm e_j\delta)=0$ for $i,j\ge 2;$
         and then
    \begin{equation}\label{eq:eiXej}
        e_i^\ast X e_1\gamma=e_i^\ast X e_j\delta=0;\qquad i,j\ge 2
    \end{equation}

    $\bullet$ Hence, if $\gamma\delta\neq0$, then $X\in (S_0\cup S_1\cup S_2)^\bot$ maps every vector into $e_1\KK$ and since also $A_1(\mu)\perp X$ and $A_1'(\mu)\perp X$ we further have $e_1^\ast Xe_1=0$ and then $e_1^\ast Xe_j=0$, giving $X=0$.

    $\bullet $ If $\delta=0$, then $|\gamma|=1$ and \eqref{eq:eiXej} reduces to $e_i^\ast Xe_1=0$ ($i\ge 2$) for $X\in(S_0\cup S_1\cup S_2)^\bot$. The matrix $A_1(\mu)$ from \eqref{eq:A1(mu)} is again in $S_1$ and hence  $A_1(\mu)\perp X$ for each unimodular $\mu$  gives $e_1^\ast Xe_1=0$. Hence $Xe_1=0$.

    Here the initial $s$-tuple is  $\{A_1=e_1e_1^\ast, A_2=(e_1\alpha+e_2\beta)e_1^\ast,\dots\}\in S_0$, and by applying isometries  $X\mapsto UX$ and $X\mapsto (UX)^\ast$, which both  fix $A_1$ (so unitary $U$ fixes $e_1$ and maps $e_2\beta$ into $e_i\beta\mu$),   we further get
    $$\{A_1, (e_1\alpha+e_i\beta \mu) e_1^\ast,\dots\}\,,\,\{A_1, e_1(e_1\alpha+e_i\beta \mu)^\ast,\dots\}\in S_1$$
    and so $e_1^\ast X e_i=0$ for $i\ge 2$. Therefore, $X\in0\oplus {\mathcal M}_{n-1}(\KK)$. We now choose unitary $U$ with $Ue_1=e_1\alpha+e_i\beta $  and apply the isometry $X\mapsto UX$ on the second of the above $s$-tuples. Notice that this isometry may fix none of its elements, but the resulting $s$-tuple
    $$\bigl\{(e_1\alpha+e_i\beta)e_1^\ast\,,\, (e_1\alpha+e_i\beta)(e_1\alpha+e_i\beta)^\ast,\dots\bigr\}$$ will still have the same properties as the original one  and will intersect the first of the above two  $s$-tuples in  a matrix $     (e_1\alpha+e_i\beta)e_1^\ast$. Thus, according to our procedure, the obtained $s$-tuple belongs to $S_2$. We now further change it by an isometry $X\mapsto XV^\ast$, where $V$ fixes $e_1$ and maps $e_i\beta$ into $ e_j\beta\mu$ to obtain an $s$-tuple
    $$\bigl\{(e_1\alpha+e_i\beta)e_1^\ast\,,\, (e_1\alpha+e_i\beta)(e_1\alpha+e_j\beta\mu)^\ast,\dots\bigr\}\in S_3$$
    Thus, $(e_1\alpha+e_i\beta)(e_1\alpha+e_j\beta\mu)^\ast \perp X$ for every unimodular $\mu$ and index $j\ge 2$. Since we already known that $Xe_1=0$ and $e_1^\ast Xe_j=0$, and since $\beta\neq0$ we get $e_i^\ast X e_j=0$ for $i,j\ge 2$, and hence $X=0$, as claimed.

    $\bullet$  If $\gamma=0$, then $|\delta|=1$.  Here, \eqref{eq:eiXej} reduces into $e_i^\ast Xe_j=0$ ($i,j\ge2$)   for $X\in(S_0\cup S_1\cup S_2)^\bot$ and      the $s$-tuple~\eqref{eq:Aij(mu)} simplifies  into
    \begin{equation}\label{eq:gamma=0}
        \bigl\{A_1= e_1e_1^\ast \,,\, A_{2j}(1)=(e_1\alpha+e_2\beta)e_j^\ast,\dots\}\in S_1;\qquad j\ge 2.
    \end{equation}
                             We  change it with isometry $X\mapsto XV^\ast$,  which fixes $A_{2j}(1)$, into
    $$\bigl\{\mu e_1e_k^\ast , A_{2j}(1),\dots\}\in S_2;\qquad k\neq j.$$
    We further apply isometry $X\mapsto X^\ast$ which fixes the first entry of  $s$-tuple \eqref{eq:gamma=0}, and then apply $X\mapsto U^\ast X$ which fixes $A_{2j}(1)^\ast$ to get an $s$-tuple
    $$\bigl\{\mu e_ke_1^\ast , A_{2j}(1)^\ast,\dots\}\in S_3;\qquad k\neq j.$$
    Again, $X\in (S_0\cup S_1\cup S_2\cup S_3)^\bot$ then satisfies $(\mu e_1e_k^\ast)\perp X$ and $(\mu e_ke_1^\ast)\perp X$. Together with \eqref{eq:eiXej} this gives  $$e_i^\ast X e_j=e_1^\ast X e_k=e_k^\ast X e_1=0 \quad \hbox{ for }i,j\ge 2 \hbox{ and }k\neq j.$$ When $n\ge 3$ this is clearly  equivalent to $X=0$.  When $n=2$ it is equivalent to  $e_k^\ast Xe_k=0$ ($k=1,2$). In this case if $\alpha\neq0$, then, by suitable isometry $X\mapsto XV^\ast$ and the adjoint map, which both fix $A_1$, we change the $s$-tuple \eqref{eq:gamma=0} into the one containing $(e_1\alpha+e_2\beta)(e_j\mu)^\ast$ and its adjoint, and it then  follows easily that  $\overline{\alpha} e_1^\ast Xe_j=e_j^\ast Xe_1\alpha=0$, that is, $X=0$.  However, if $n=2$ and $\alpha=0$ then, as explained at the beginning (and since $\gamma=0$), ${\boldsymbol A}=\{ E_{11},\mu_2 E_{22},\dots\}\subseteq \bigcup_{i,j} \KK e_ie_j^\ast \subseteq {\mathcal M}_2(\KK)$. 
    Notice that   the isometry $X\mapsto \left(\begin{smallmatrix}
        \mu & 0\\
        0 & 1
    \end{smallmatrix}\right)X$  fixes the second element of ${\boldsymbol A}$   and transforms ${\boldsymbol A}$  into $\{\mu E_{11},\mu_2 E_{i_0j_0},\dots\}\in S_2$. We easily deduce that either $\mathfrak{S}\subseteq\left(\begin{smallmatrix}
        0 & \KK\\
        0 & \KK
    \end{smallmatrix}\right)$, where every element is left-symmetric so $ \bigcap_{Z\in\mathcal L_{\mathfrak S}} Z^\bot
    =0$, or  $\mathfrak{S}\subseteq\left(\begin{smallmatrix}
        0 & \KK\\
        \KK & 0
    \end{smallmatrix}\right)$.
    In the latter case we reduce to one of previous (sub)cases except when  every $\boldsymbol X\in\bigcup S_i$ is contained in $\KK E_{11}\cup \KK E_{22}$ in which case $\mathfrak{S}=\left(\begin{smallmatrix}
        0 & \KK\\
       \KK& 0
    \end{smallmatrix}\right) $
    and  each element lying in a single block of $\mathfrak{S}$ is left-symmetric. Hence, yet again
          $\bigcap_{Z\in\mathcal L_{\mathfrak S}} Z^\bot
         =0 $ which completely verifies~\eqref{eq987654}.
    \end{proof}

    \textit{Proof of Theorem \ref{maintheorem1}}. Let $\phi\colon \mathcal A_1\rightarrow\mathcal A_2$ be a BJ-isomorphism and let $\mathcal A_1$ be  finite-dimensional  with  $\dim \A_1\ge2$.   Then, by  \cite[equation~(1.2)]{guterman},  $\mathop{\mathrm{dim}}\mathcal A_1=\mathop{\mathrm{dim}}\mathcal A_2$. If $\A_1$ is a pseudo-abelian $C^\ast$-algebra, then as proved in \cite[Theorem 1.1]{abelian},  $\A_1$ and $\A_2$ are $C^*$-isomorphic with $\FF_1=\FF_2$. So, assume $\A_1$ and $\A_2$ are not pseudo-abelian $C^*$-algebras.

    Without loss of generality, let $\mathcal A_1=\bigoplus_{i=1}^{\ell_1} \mathcal M_{n_i^1}(\mathbb K_i^1)$ and $\mathcal A_2=\bigoplus_{j=1}^{\ell_2} \mathcal M_{n_j^2}(\mathbb K_j^2)$.
         Now since $\phi$ is BJ-isomorphism, we have  \begin{equation}\label{isomorphism1}\phi(\mathcal L_{\A_1}^{\bot})=\mathcal L_{\A_2}^{\bot}\quad \text{ and }\quad \phi(\mathcal L_{\A_1}^{\bot\bot})=\mathcal L_{\A_2}^{\bot\bot}.\end{equation} By \cite[Theorem 1.2]{abelian}, we have that  $\mathcal L_{\A}^{\bot\bot}$
    is the pseudo-abelian summand (i.e., a  sum of its $1$-by-$1$ matrix block constituents) in any finite-dimensional $C^*$-algebra $\A$, while  $\mathcal L_{\A}^{\bot}$ is the nonpseudo-abelian summand (i.e., a  sum of its $n$-by-$n$ matrix blocks with $n\geq 2$).

    Therefore, by  \eqref{isomorphism1}, $\phi$ induces  BJ-isomorphism between pseudo-abelian summands of $\A_1$ and $\A_2$ and it also induces BJ-isomorphism between their nonpseudo-abelian summands. 
    
    Since we assume $\A_1$ and $\A_2$ are not pseudo-abelian $C^*$-algebras, then $\mathcal B_1=\mathcal L_{\A_1}^\bot$ and $\mathcal B_2=\mathcal L_{\A_2}^\bot$ are non-zero and BJ-isomorphic $C^*$-subalgebras. Using Lemmas~\ref{coordinatewisesymmetricity} and  \ref{minimalsmooth}, there exists a minimal number $s\ge1$ such that we can find $s$ smooth elements $A_1,\dots, A_s\in\mathcal B_1$ such that $\bigcap_i A_i^\bot$ contains a non-zero left-symmetric element inside $\mathcal B_1$. Now, smooth points  and left-symmetricity are both defined in terms of BJ orthogonality so  are both preserved under BJ-isomorphism. Therefore,  $\phi(A_1),\dots,\phi(A_s)\in{\mathcal B_2}$ are smooth points  and $\bigcap_i \phi(A_i)^\bot=\phi(\bigcap_i A_i^\bot)$ contains a non-zero left-symmetric element (in $\mathcal B_2$). The same arguments work for $\phi^{-1}$ and   prove that $s$ is also the minimal number of smooth points needed in $\mathcal B_2$ which satisfies this condition.

    Now applying the procedure mentioned in Lemma \ref{procedure}, which is  defined  completely in terms of BJ orthogonality alone (and in view of Remark \ref{remark2.8}) there exist indices $j_1,j_2$ such that  $\phi$ maps
    $$0\oplus {\mathcal M}_{n_{j_1}^1}(\KK_{j_1}^1)\oplus 0=\bigg(\mathfrak S\cap\big( \bigcap_{Z\in\mathcal L_{\mathfrak S}} Z^\bot
         \big)\cap \big(\bigcap_{Z\in\mathcal L_{\boldsymbol{A}^\bot}} Z^\bot\big)\bigg)^
    \bot\subseteq{\mathcal B_1}$$ bijectively onto $0\oplus {\mathcal M}_{n_{j_2}^2}(\KK_{j_1}^2)\oplus 0\subseteq{\mathcal B_2}$. Using \cite[Theorem 1.1]{simple}, we get that $\FF_1=\FF_2$ and $\mathcal M_{n_{j_1}^1}(\mathbb K_{n_{j_1}}^1)=\mathcal M_{n_{j_2}^2}(\mathbb K_{n_{j_2}}^2)$.

         Further, $\phi$ maps $\mathcal B_1'=
    (0\oplus\mathcal M_{n_{j_1}^1}(\mathbb K_{j_1}^1)\oplus0)^\bot\subseteq \mathcal B_1$ bijectively onto $\mathcal B_2'=
    (0\oplus\mathcal M_{n_{j_2}^2}(\mathbb K_{j_1}^2)\oplus0)^\bot\subseteq \mathcal B_2$. If $\mathcal B_1'=0$, then $\mathcal B_2'=0$ and we get that the nonpseudo-abelian summands of $\mathcal A_1$ and $\mathcal A_2$ are same. Otherwise, $\phi$ induces a BJ-isomorphism between $C^*$-algebras $\mathcal B_1'$ and $\mathcal B_2'$.

    This allows us to apply the procedure from Lemma \ref{procedure} on $\mathcal B_1'$ and $\mathcal B_2'$ to extract one further block from the matrix block decomposition of $\mathcal B_1'$, which also belongs $\mathcal B_2'$. Then, proceeding recursively, we get the blocks in the matrix block decomposition of $\mathcal B_1$ and $\mathcal B_2$ are same, upto permutation.
              It implies that nonpseudo-abelian summands of $\A_1$ and $\A_2$ are $\ast$-isomorphic.

    Finally, since, $\mathbb F_1=\mathbb F_2$, using \cite[Theorem 1.1]{abelian}, the pseudo-abelian summands of $\A_1$ and $\A_2$ are also $\ast$-isomorphic. It implies $\A_1$ and $\A_2$ are $\ast$-isomorphic.
         \qed

    The above proof of Theorem \ref{maintheorem1} shows even more, which we state in terms of minimal ideals and their centralizers (see a book by Grove~\cite{grove}). The minimal ideals in \eqref{eq:matrixdecompositionC} and \eqref{eq:matrixdecompositionR} are the matrix blocks ${\mathcal M}_{n_i}(\KK_i)$ and their centralizers are~$\KK_i$.

         \begin{corollary}
        Let $\phi\colon \A_1\rightarrow\A_2$ be a BJ-isomorphism between $C^*$-algebras $\A_1$ and $\A_2$ over the fields $\FF_1$ and $\FF_2$, with $2\le \dim \A_1<\infty$. Then, by identifying $\A_1=\A_2$, $\phi$ permutes the  minimal ideals which share the same centralizer and  dimension.
                   \end{corollary}

    \begin{proof} In the proof of Theorem \ref{maintheorem1}, we already showed  this  for nonpseudo-abelian summands of $\A_1$ and $\A_2$ (that is $\phi$ induces a permutation among  the blocks of size bigger than one in \eqref{eq:matrixdecompositionC}--\eqref{eq:matrixdecompositionR}). Furthermore, we also  showed that pseudo-abelian summand of $\A_1$ is mapped onto pseudo-abelian summand of $\A_2$. So, without loss of generality we can assume $\A_1$ and $\A_2$ are pseudo-abelian $C^\ast$-algebras. We now give a procedure (in parallel to Lemma~\ref{procedure}) which will  extract all elements within a single  block of a  finite-dimensional pseudo-abelian $C^*$-algebra $\A$ and will be  defined completely in terms of BJ orthogonality.

    To start, let $s\ge0 $ be the minimal number such that there exists non-zero left symmetric points $A_1,\dots, A_s$  such that $\{A^\bot;\;\; A\in\mathcal L_{\bigcap_{i=1}^sA_i^\bot}\}$ is a finite set (see \cite[Lemma 3.5]{abelian}).  If $s=0$, then by \cite[Lemma 3.2]{abelian} $\A=\FF^n$ and left-symmetric elements with the same outgoing neighbourhoods constitute the same block.

    Assume $s\ge 1$. Then, $\A$ is a real $C^*$-algebra and $\A=\RR^r\oplus\CC^c\oplus\HH^h$ with $(c,h)\neq(0,0)$,  each left-symmetric $A_i$ belongs to one of the complex or quaternionic blocks, and $\bigcap_{i=1}^sA_i^\bot$ is isometrically isomorphic to $\RR^{r+c+h}$ (see \cite[Lemmas 3.2 and 3.5]{abelian}).
    In this case, we remove $A_1$ and collect all left-symmetric $X$ such that $\{A^\bot;\;\; A\in\mathcal L_{X^\bot\cap \big(\bigcap_{i=2}^sA_i^\bot\big)}\}$ is a finite set.
         It follows from \cite[Lemma 3.5]{abelian} that all these $X$ share the same  block with $A_1$. Remove, if possible, $A_i$ for $i\geq 2$ such that $\{A^\bot;\;\; A\in\mathcal L_{X^\bot\cap A_1^\bot\cap(\bigcap_{j\not\in\{1,i\}}A_j^\bot)}\}$ is a finite set. If, for no $X$ from our collection, there is  such $A_i$, then~$A_1$ belongs  to a complex block by \cite[Lemma 3.4]{abelian}. Otherwise, all these $A_i$  share the same (quaternionic) block with $A_1$, again by \cite[Lemma 3.5]{abelian}. Collect all such $A_i$ and the corresponding $X$ obtained when the above procedure is started with  one of  them in place of $A_1$. It is an exercise that this collection    coincides with all  the non-zero elements from  the block containing~$A_1$ (see~\cite[Lemma 3.4 and its~proof]{abelian}).

    After applying this procedure we get a collection $\Omega$ consisting of all the non-zero elements from one block of $\A$. To extract other blocks, we apply it again  on a $C^*$-subalgebra $\Omega^\bot$ and proceed recursively.
                   \end{proof}

    \end{document}